\documentclass[reqno]{amsart}
\usepackage[utf8]{inputenc}
\usepackage[T2A]{fontenc}
\usepackage[russian,english]{babel}
\makeatletter
\def\@settitle{\begin{center}%
\baselineskip14\p@\relax
\bfseries
\@title
\end{center}%
}
\makeatother
\usepackage{amsthm, amsmath, amssymb}
\usepackage{comment}
\usepackage{mathrsfs}
\usepackage{graphicx}
\usepackage{varioref}
\usepackage{cleveref}
\usepackage{xy}
\xyoption{all}
\usepackage[enableskew]{youngtab}
\usepackage{multirow}
\usepackage{comment}
\usepackage{cite}

\theoremstyle{theorem}
\newtheorem{theo}{Theorem}[section]
\newtheorem{lemma}[theo]{Lemma}

\theoremstyle{definition}
\newtheorem{defi}[theo]{Definition}

\theoremstyle{remark}
\newtheorem{remark}[theo]{Remark}

\usepackage{ifpdf}
\ifpdf
	\makeatletter
		\g@addto@macro\Gin@extensions{,.eps}
	\makeatother
\fi
\title[]{Hirzebruch functional equations \\ and complex Krichever genera}

\author{I.\,V.\,Netay}

\thanks{This work is supported by the Russian Science Foundation under grant 14-50-00005.}
\address{Steklov Mathematical Institute of Russian Academy of Sciences}

\begin{document}
\maketitle

\begin{abstract}
	It is well known that the two-parametric Todd genus and elliptic functions of level~$d$ define $n$-multiplicative Hirzebruch genera, if~$d$ divides~$n+1$.
	Both these cases are particular cases of Krichever genera defined by the Baker--Akhiezer functions.
	In this work the inverse problem is solved.
	Namely, it is proved that only these families of functions define $n$-multiplicative Hirzebruch genera among all the Krichever genera for all~$n$.
\end{abstract}

\tableofcontents

\section{Introduction}

\subsection{Motivation}
	The paper is devoted to study of solutions of functional equation
	\begin{equation}
		\label{eq:H}
		\sum_{k=0}^n\prod_{i\ne k} \frac{1}{f(u_i-u_k)} = C.
	\end{equation}
	with initial data~$f(0)=0$,~$f'(0)=1$ known as {\it Hirzebruch functional equation}.
	The following topological interpretation is well known (see details in~\cite[\S4]{BucNet15}).
	For any formal series~$f(t)=t+\ldots$ over a coefficient ring~$R$ one can construct {\it complex Hirzebruch genus}~$L_f$.
	It defines a homomorphism~$L_f\colon \Omega_U\to R$, where~$\Omega_U$ is the cobordism ring of stably complex manifolds.

	Hirzebruch genus~$L_f$ is called {\it fiberwise multiplicative} with respect to bundles of stably complex manifolds with fiber~$F$, if for any such a bundle~$E\to B$ with fiber~$F$ we have
	$
		L_f(E) = L_f(B) L_f(F).
	$
	If the fiber~$F=\mathbb{P}_\mathbb{C}^n$ is the complex projective space of dimension~$n$, this condition is called {\it$n$-multiplicativity} and is equivalent to equation~\eqref{eq:H}.

	It is well known that the function
	\begin{equation}
		\label{f:Todd}
		f_{\rm Td}(t) := \frac{e^{\alpha t}-e^{\beta t}}{\alpha e^{\beta t}-\beta e^{\alpha t}}
	\end{equation}
	defines a $n$-multiplicative Hirzebruch genus for each~$n$.
	It is called two-parametric Todd genus and naturally appears in many branches of mathematics, for example, in enumerative combinatorics and toric topology.
	We will see below (see~\cref{lem:Td}) that it satisfies~\eqref{eq:H}.
	It is well known that a function~$f(t)$ determines a multiplicative Hirzebruch genus if and only if it has the form~\eqref{f:Todd} (see~\cite{Mus11}).

	Other interesting families of $n$-multiplicative genera can be constructed in terms of elliptic functions.
	Suppose~$\Lambda=2\pi i(\mathbb{Z}\oplus\tau\mathbb{Z})\subset\mathbb{C}$ is a lattice, where~$\operatorname{Im}\tau>0$.
	There exists a $\Lambda$-periodic function~$g(t)$ with zero of order~$d$ at~$0$, pole of order~$d$ at~$P$ and without other zeros and poles modulo~$\Lambda$ for any some~$d>1$.
	From theory of elliptic functions it is known that this function is uniquely defined under the normalization condition~$g(t)=t^d+\ldots$, and in this case~$P=2\pi i(k+l\tau)/d$ for~$k,l\in\mathbb{Z}$.
	Under the normalization condition it is uniquely defined by the choice of~$P$.

	Under these assumptions we can extract the unique root of degree~$d$ with the normalization condition
	\[
		f_d(t):=\sqrt[d]{g(t)} = t+ \ldots
	\]
	The function~$f_d(t)$ is periodic with respect to the lattice~$d\Lambda$.
	Such functions are called {\it elliptic of level~$d$}.
	They were introduced in~\cite{Hirz88}.

	It is know that elliptic functions of level~$d$ are solutions of~\eqref{eq:H}, if~$d$ divides~$n+1$ (see the proof in~\cite{BucNet15} or~\cite{HBJ}).

\subsection{Main results}
	Some recent papers (see~\cite{BucNet15,BucBun14,BucBun15}) are devoted to classification of solutions of equations~\eqref{eq:H} for different~$n$.
	In~\cite{BucNet15} it is shown that for each~$n>1$ solutions of equation~\eqref{eq:H} are parametrized by points of some affine algebraic variety, if the coefficient ring~$R$ is a field of zero characteristic.
	It is also shown there how to find polynomial equations defining it.
	In~\cite{BucBun15} this variety is described for~$n=2$ with the corresponding functions being solutions of~\eqref{eq:H}, in~\cite{BucNet15} the same for~$n=3$.

	In this work we do not fix any~$n$, but restrict the class of functions.
	Namely, we consider Baker--Akhiezer functions (see the definition in~\cref{sec:sketch and notation}), defining complex Krichever genera.
	These genera were introduced in~\cite{Kri90} and contain listed above as particular and limit cases.
	Our goal is to classify all $n$-multiplicative genera corresponding to Baker--Akhiezer functions for all~$n$.
	The main result is the following Theorem.

	\begin{theo}
		If a function~$f_{\rm Kr}(t)$ defined in~\eqref{f:Kr} and defining a complex Krichever genus is a solution of~\eqref{eq:H} for some~$n$, then one of the following cases holds:
		\begin{itemize}
			\item[(i)] $f_{\rm Kr}(t)=f_{\rm Td}(t)$ for some values~$\alpha$ and~$\beta$;
			\item[(ii)] $f_{\rm Kr}(t)=f_{d}(t)$ is an elliptic function of level~$d$, where~$d$ divides~$n+1$.
		\end{itemize}
	\end{theo}

	Applied methods here are similar to~\cite{BucNet15}, but there are some differences.
	In~\cite{BucNet13} we apply ideas of symmetric functions to an arbitrary formal series defining a Hirzebruch genus.
	Here we restrict the class of functions.
	In the general case we consider some particular symmetric functions and equations on their coefficients.
	Here we can consider all the equations and obtain the complete classification in this family of functions.

\section{Proofs of main results}

\subsection{Sketch of proof and some needed formulas}
\label{sec:sketch and notation}
	Let~$\Lambda\subset \mathbb{C}$ be a lattice and~$\sigma$ and~$\zeta$ be the standard corresponding Weierstrass functions.
	Then the Baker--Akhiezer function is
	\begin{equation}
		\label{f:BA}
		\Phi_s(t) := \frac{\sigma(s - t)}{\sigma(s)\sigma(t)} e^{\zeta(s)t}.
	\end{equation}
	For this function the addition low holds:
	\begin{equation}
		\label{eq:Phi add}
		\Phi_s(x+y) (\wp(y) - \wp(x)) = \Phi_s'(x)\Phi_s(y)  - \Phi_s'(y)\Phi_s(x).
	\end{equation}
	Applying to this another well known relation
	\[
		(\ln \Phi_s(t))' = \frac{1}{2} \frac{\wp'(t)-\wp'(s)}{\wp(t)-\wp(s)},
	\]
	one can rewrite in the form
	\begin{equation}
		\label{eq:Phi add 2}
		\frac{\Phi_s(x+y)}{\Phi_s(x)\Phi_s(y)} = -\frac{1}{2}
		\frac{
			\begin{vmatrix}
				1 & 1 & 1 \\
				\wp(x) & \wp(y) & \wp(s) \\
				\wp'(x) & \wp'(y) & \wp'(s) \\
			\end{vmatrix}
		}{
			\begin{vmatrix}
				1 & 1 & 1 \\
				\wp(x) & \wp(y) & \wp(s) \\
				\wp^2(x) & \wp^2(y) & \wp^2(s) \\
			\end{vmatrix}
		}
	\end{equation}
	We will also need the formula
	\begin{equation}
		\label{eq:Phi minus}
		\Phi_s(x)\Phi_s(-x) = \wp(s) - \wp(x).
	\end{equation}
	These formulas are well known and can be found in~\cite{Kri80,WW}.

	\begin{defi}
		{\it Krichever genus} is a Hirzebruch genus determined by the series expansion of the function
		\begin{equation}
			\label{f:Kr}
			f_{\rm Kr}(t) := \frac{e^{\alpha t}}{\Phi_s(t)}.
		\end{equation}
	\end{defi}


	If the elliptic curve degenerates, then the function~\eqref{f:Kr} becomes the following (see.~\cite{Kri90})
	\begin{equation}
		\label{f:BAd}
		f_{\rm sing}(t) = \frac{e^{\lambda t}}{q + \eta\cth \eta t}.
	\end{equation}

	The methods for classification of $n$-rigid degenerate and non-degenerate Krichever genera are essentially different.
	In the case of non-degenerate Krichever genera we apply methods of elliptic function theory and show that only the elliptic functions of level $d$, where $d$ divides~$n+1$, satisfy the property of  $n$-multiplicativity.
	In the case of degenerate Krichever genera the proof is more complicated and applies methods of symmetric functions.
	Technically it is similar to the proof of classification of $3$-rigid genera in~\cite{BucNet15}.
	But this proof requires the computations in the ring of symmetric meromorphic functions, but not symmetric polynomials.
	In this case the classification includes the two-parametric Todd genus and degenerations of elliptic functions of level~$d$.

	Let us recall some basic facts about symmetrical polynomials and introduce some notation.

	Let~$\lambda = \{\lambda_1\geqslant\ldots\geqslant\lambda_k\}$ be a Young diagram.
	Set
	\[
		\Delta_\lambda =
		\begin{vmatrix}
			x_1^{\lambda_1} & \ldots & x_n^{\lambda_n} \\
			\vdots & \ddots & \vdots \\
			x_1^{\lambda_n} & \ldots & x_n^{\lambda_n} \\
		\end{vmatrix}.
	\]
	Denote~$\delta=(n-1,\ldots,1,0)$.
	It is easy to see that~$\Delta_{\lambda+\delta}$ is divisible by~$\Delta_{\delta}$ for any Young diagram~$\lambda$.

	\begin{defi}
		{\it The Schur polynomial} corresponding to a diagram~$\lambda$ is defined by the formula
		\[
			s_\lambda = \Delta_{\lambda+\delta} / \Delta_\delta.
		\]
	\end{defi}

	Now let us weaken some assumptions to get some appropriate construction for our computations.
	Namely, we omit the assumption that~$\lambda_i \in \mathbb{Z}$.

	\begin{defi}
		Let~$\lambda=\{\lambda_1,\ldots,\lambda_n\}$ be a multiindex.
		We call the {\it Schur--Puiseux function} the expression
		\[
			s_\lambda = \Delta_{\lambda+\delta} / \Delta_\delta.
		\]
	\end{defi}

	After this definition we need to clarify some points to avoid confusion after.
	At first, we do not assume that~$\lambda_i \in \mathbb{Q}$ as it is assumed for Puiseux series.
	We consider~$\lambda_i\in\mathbb{C}$, but interesting for us cases occur only for~$\lambda_i\in\mathbb{Q}$.
	At second, these expressions are not actually polynomials and are not even linear combinations of "Puiseux monomials" in~$x_0,\ldots,x_n$ in general.
	Nevertheless, the construction is the same, and the basic ideas remain the same.

	Let us introduce some differential operators.
	Denote~$\delta_i = (x_i - x_{i-1})^{-1}(1 - \tau_i)$, where~$\tau_i$ acts as a permutation of variables~$x_{i-1}$ and~$x_i$ on the space of polynomials~$R[x_0,\ldots,x_n]$.
	It is well known that~$\delta_i(P)$ is a polynomial, if~$P$ is a polynomial.
	It is easy to see that
	\begin{itemize}
		\item $\delta_i(P)$ is invariant under~$\tau_i$,
		\item $\delta_i(P)=0$ if and only if~$P$ is~$\tau_i$-invariant.
		\item $\delta_i$ is a $\tau_i$-derivation on the ring of polynomials as a module over the subring of symmetric polynomials:
		\[
			\delta_i(P\cdot Q) = \delta_i(P)\cdot Q + \tau_i(P)\cdot\delta_i(Q).
		\]
	\end{itemize}

	Consider the differential operator
	\[
		L = \frac{1}{\Delta} \sum_{\sigma \in \mathfrak{S}_{n+1}}(-1)^\sigma\sigma,
	\]
	where~$\mathfrak{S}_{n+1}$ denotes the permutation group of the variables~$x_0,\ldots,x_n$ and~$\Delta=\Delta_\delta(x_0,\ldots,x_n) = \prod_{i<j}(x_i-x_j)$.
	We use~$x_i$ instead of~$u_i$ here, because we will apply~$L$ for~$x_i=e^{2u_i}$.
	It is well known that (see~\cite{FultonYoungTableaux})
	\[
		L = (\delta_1\ldots\delta_n)(\delta_1\ldots\delta_{n-1})\ldots(\delta_1\delta_2)\delta_1.
	\]

	Clearly, $L(1)=0$ and~$L(P\cdot Q)=L(P)\cdot Q$ for any symmetric polynomial~$Q$.
	Also for any symmetric~$Q$ the following important relation holds:
	\begin{equation}
		\label{eq:L}
		L(\Delta Q) = (n+1)!Q.
	\end{equation}

	The action of~$L$ can be extended to the field of rational functions in~$x_0,\ldots,x_n$.
	The extension of~$L$ possesses all the properties above.
	In the same way it can be extended to the field of Puiseux series~$\mathbb{C}((x_0^{1/N},\ldots,x_n^{1/N}))$ for each~$N$.
	We need to extend it to the field of meromorphic function in~$u_0,\ldots,u_n$.
	We will set~$x_i=e^{2u_i}$, so the powers~$x_i^\lambda$ correspond to the functions~$e^{2\lambda u_i}$.
	Obviously, the operator~$\frac{1}{(n+1)!}L\Delta$ is also a projection from the field of meromorphic functions to the field of symmetric meromorphic functions.

\subsection{Non-degenerate Krichever genera}

	\begin{lemma}
		For function~$f(t)$ of the form~\eqref{f:Kr} corresponding to non-degenerate elliptic curve the equation~\eqref{eq:H} is equivalent to an equation of the form
		\begin{equation}
			\label{eq:Hab}
			\sum_{k=0^n}
			\frac{\mathcal{A}_k}{f^{n+1}(u_k)} = C\prod_{k=0}^n \frac{1}{f(u_k)},
		\end{equation}
		where~$\mathcal{A}_k$ are non-zero abelian functions in~$u_0,\ldots,u_n$.
	\end{lemma}

	\begin{remark}
		The explicit expression for~$\mathcal{A}_k$ can be extracted from the proof, but it is not so important as the fact that it is an abelian function.
	\end{remark}

	\begin{proof}
		Let us multiply equation~\eqref{eq:H} by~$\prod_{k=0}^n\frac{1}{f(u_k)}$ and obtain
		\[
			\sum_{k=0}^n \frac{1}{f^{n+1}(u_k)} \prod_{i\ne k} \mathcal{B}_{i,k} = C\prod_{k=0}^n \frac{1}{f(u_k)},
		\]
		where~$\mathcal{B}_{i,k} = \mathcal{B}_{i,k}(u_i,u_k) = \frac{\Phi(u_k-u_i)}{\Phi(u_k)}\Phi(u_i)$.
		From the properties of Baker--Akhiezer function it follows that functions~$\mathcal{B}_{i,k}$ are abelian in~$u_i$ and~$u_k$.
		This can be directly shown using the addition law for Baker--Akhiezer function~(see~\eqref{eq:Phi add}):
		\begin{multline*}
			\mathcal{B}_{i,k} =
			\frac{\Phi(u_k-u_i)}{\Phi(u_k)}\Phi(u_i) =
			\frac{1}{\wp(s)-\wp(u_i)} \frac{\Phi(u_k - u_i)}{\Phi(u_k)\Phi(-u_i)} =\\=
			-\frac{1}{2(\wp(s)-\wp(u_i))}
			\frac{
				\begin{vmatrix}
					1 & 1 & 1 \\
					\wp(u_k) & \wp(u_i) & \wp(s) \\
					\wp'(u_k) & -\wp'(u_i) & \wp(s) \\
				\end{vmatrix}
			}{
				\begin{vmatrix}
					1 & 1 & 1 \\
					\wp(u_k) & \wp(u_i) & \wp(s) \\
					\wp^2(u_k) & \wp^2(u_i) & \wp^2(s) \\
				\end{vmatrix}
			}.
		\end{multline*}
		Therefore functions~$\mathcal{A}_k = \prod_{i\ne k}\mathcal{B}_{i,k}$ are abelian in all variables~$u_0,\ldots,u_n$.
	\end{proof}

	\begin{theo}
		Suppose that a non-degenerated Baker--Akhiezer function~$f(t)$ is a solution of Hirzebruch $n$-equation.
		Then in this equation~$C=0$, and~$f(t)$ is an elliptic function of level~$d$, where~$d$ divides~$n+1$.
	\end{theo}

	\begin{proof}
		Let~$\omega_1,\omega_2,\omega_3$ be the half-periods of elliptic curve corresponding to the function~$f(t)$.
		Under the argument shift by period~$2\omega_l$ the function~$f(t)$ multiplies by the factor~$\varepsilon_l$:
		\[
			\frac{1}{f(t+2\omega_l)} = \frac{\varepsilon_l}{f(t)}.
		\]
		This factor~$\varepsilon_l$ depends only on the choice of period~$2\omega_l$ and does not depend on~$t$.
		It can be seen from the explicit form~\eqref{f:BA} of the Baker--Akhiezer function and from the properties of~$\sigma$-function.
		We can write it down explicitly, but we do not need its explicit form.
		Actually, the statements to prove is that~$\varepsilon_l^{n+1}=1$ for all~$l$ and that~$C=0$, if the elliptic curve is non-degenerate.

		Suppose that~$C=0$.
		Note that only the first summand in the left-hand side of~\eqref{eq:Hab} depends on~$u_0$.
		Consider the shift of~$u_0$ by the period~$2\omega_l$ for some~$l$ and subtract from the result the equation~\eqref{eq:Hab}:
		\[
				(\varepsilon_l^{n+1} - 1) \frac{\mathcal{A}_k}{f^{n+1}(u_0)} = 0.
		\]
		Since the ratio is a non-zero function, we obtain~$\varepsilon_l^{n+1} = 1$ for each~$l$.

		Assume that~$C\ne 0$.
		Consider the shift of~$u_0$ by the period~$2\omega_l$ for some~$l$ and subtract from the result the equation~\eqref{eq:Hab}:
		\begin{equation}
			\label{eq:HCne0}
			(\varepsilon_l^{n+1} - 1) \frac{\mathcal{A}_k}{f^{n}(u_0)} = C(\varepsilon_l-1) \prod_{i\ne 0}\frac{1}{f(u_i)}.
		\end{equation}
		Right-hand side of~\eqref{eq:HCne0} is non-zero and does not depend on~$u_0$.
		At the same time left-hand side multiplies by~$\varepsilon_l^n$ under the shift of~$u_0$ by~$2\omega_l$.
		Therefore~$\varepsilon_l^n=1$ and~$f^n(t)$ is abelian.
		Denote~$\widetilde{\mathcal{A}}_k = \mathcal{A}_k/f^n(u_k)$.
		It is also a non-zero abelian function in~$u_0,\ldots,u_n$.
		Equation~\eqref{eq:Hab} can be rewritten in the following form:
		\[
			\sum_{k=0}^n\frac{\widetilde{A}_k}{f(u_k)} = C\prod_{k=0}^n\frac{1}{f(u_k)}.
		\]
		Shifting~$u_0$ by~$2\omega_l$ and subtracting the initial equation, we obtain
		\[
			\frac{1}{f(u_0)}(\varepsilon_l-1)
			\left(
				\widetilde{\mathcal{A}}_0 - C\prod_{k\ne 0}\frac{1}{f(u_k)}
			\right) = 0.
		\]
		Note that~$\widetilde{\mathcal{A}}_0$ is abelian in~$u_0,\ldots,u_n$ and the product has the unique pole at zero in the period parallelogram with respect to variable~$u_0$ for generic values of other~$u_i$.
		Therefore the expression in the brackets is a non-zero function.
		Hence~$\varepsilon_l=1$ for each~$l$, and~$f(t)$ is an abelian function.
		But it has the unique simple pole in the period parallelogram.
		We come to the contradiction in the case~$C\ne 0$.
	\end{proof}

\subsection{Degenerate Krichever genera}

	\begin{lemma}
		\label{lem:Td}
		The function $f_{\rm Td}(t)$ defined in~\eqref{f:Todd} satisfies~\eqref{eq:H} for each~$n$.
	\end{lemma}

	\begin{proof}
		The following relation holds:
		\[
			\frac{e^{\alpha t}-e^{\beta t}}{\alpha e^{\beta t}-\beta e^{\alpha t}} =
			\frac{1}{k + \eta \cth \eta t},
		\]
		where~$\alpha = \eta - q$ and~$\beta = -\eta - q$.
		The following addition law holds:
		\[
			\cth(x+y) = \frac{\cth(x)\cth(y) + 1}{\cth(x) + \cth(y)}.
		\]
		Let us substitute it:
		\[
			\sum_{k=0}^n\prod_{i\ne k}(k+\eta\cth(\eta u_i-\eta u_j))=\sum_{k=0}^n\prod_{i\ne k}\left(k+\eta\frac{1 - \cth(\eta u_i)\cth(\eta u_k))}{\cth(\eta u_k)-\cth(\eta u_i)}\right).
		\]
		Let is multiply it by~$\prod_{i<j}(\cth(\eta u_i)-\cth(\eta u_j))$:
		\[
			\sum_{k=0}^n\prod_{i\ne k}(k\cth\eta u_k - k\cth\eta u_i + \eta - \eta\cth\eta u_i\cth\eta u_k)\prod_{i<j\atop i,j\ne k}(\cth\eta u_i-\cth\eta u_j).
		\]
		Denote~$q_i=\cth\eta u_i$ and obtain
		\[
			\sum_{i\ne 0}\prod_{i\ne k}(kq_i-kq_k+\eta-\eta q_iq_k)\prod_{i<j\atop i,j\ne k}(q_i-q_j).
		\]
		This polynomial is skew-symmetric in the variables~$q_i$.
		There exists the unique up to proportionality skew-symmetric polynomial in $q_i$, $i=0,\ldots,n$, such that the degree of each the variable in each monomial does not exceed $n$:
		\[
			\prod_{i<j}(q_i-q_j).
		\]
		Therefore the equality
		\[
			\sum_{k=0}^n\prod_{i\ne k}(q+\eta\cth\eta(u_k-u_i)) = C
		\]
		holds for each $n$ with a constant $C$ depending on $n$.
	\end{proof}

	\begin{theo}
		Suppose that a degenerate Baker--Akhiezer function~$f_{\rm sing}(t)$ is a solution of Hirzebruch $n$-equation.
		Then one of the following cases holds:
		\begin{itemize}
			\item $\lambda = 0$,~i.\,e.~$f(t)$ has form~\eqref{f:Todd};
			\item $f(t) = \frac{1}{\mu}e^{\lambda\mu t}\sinh\mu t$, where~$\lambda=\frac{2k}{n+1}$ for~$k\in\mathbb{Z}$.
		\end{itemize}
	\end{theo}

	\begin{proof}
		One of two cases holds: either~$\mu\ne 0$ in~\eqref{f:BAd} or the limit case
		\[
			f_{\rm sing}^\circ := \lim_{\mu\to 0}f_{\rm sing} = \frac{te^{\lambda t}}{1+qt}.
		\]
		Clearly,~$f_{\rm sing}^\circ$ satisfies~\eqref{eq:H} for~$\lambda=0$.
		For~$\lambda\ne 0$ the expression $\sum_{k=0}^n\prod_{i\ne k} e^{\lambda(u_i-u_k)}(q+u_k-u_i)$ is a combination of non-trivial exponents with non-trivial polynomial coefficients.
		Obviously, it is not a constant function.
		Therefore, for~$\lambda\ne 0$ equation~\eqref{eq:H} does not hold.
		So we assume after that~$\mu\ne 0$.

		Note that applying the replacement~$f_{\rm sing}(t) \mapsto \mu f_{\rm sing}(t/\mu)$, $q\mapsto q\mu$, $\lambda\mapsto \lambda\mu$, we can pass from~\eqref{f:BAd} to the function
		\begin{equation}
			\label{f:BAd'}
			f_{\rm sing}(t) = \frac{e^{\lambda t}}{q + \cth t}.
		\end{equation}
		Under this replacement~$C$ in~\eqref{eq:H} should be replaced with~$\mu^nC$.

		\begin{gather*}
			\intertext{Let us substitute the function~\eqref{f:BAd'} into~\eqref{eq:H}:}
			\sum_{k=0}^n\prod_{i\ne k}(q + \cth(u_k-u_i))e^{\lambda(u_k-u_i)} =
			\sum_{k=0}^n\prod_{i\ne k}\left( q + \frac{e^{2u_k}+e^{2u_i}}{e^{2u_k} - e^{2u_i}}\right)\frac{e^{\lambda u_k}}{e^{\lambda u_i}} = C.
			\intertext{Let us denote~$x_i^\lambda = e^{2\lambda u_i}$.
			Multiplying this by }
			\label{f:xldelta}
			\prod\limits_{i=0}^n e^{\lambda u_i}\prod\limits_{i<j}(e^{2u_i} - e^{2u_j}) = x_0^{\lambda/2}\ldots x_n^{\lambda/2}\Delta,
			\intertext{where~$\Delta=\Delta_\delta(x_0,\ldots,x_n)$, we get}
			\sum_{k=0}^n(-1)^{kn} x_k^{\lambda(n+1)/2}\prod_{i\ne k}(ax_k + bx_i)\prod_{i < j \atop i,j\ne k}(x_i-x_j) =
			C \prod_{i=0}^n x_i^{\lambda/2} \prod_{i<j}(x_i - x_j),
			\intertext{where~$a = 1+q$ and $b = 1-q$.
			Note that in the replacement above we can take~$-\mu$ instead of~$\mu$.
			This corresponds to the swap of~$a$ and~$b$.
			Now let us apply the operator~$L$ and the equality~\eqref{eq:L} to this equation.
			The right-hand side gives }
			L\left( C (x_0\ldots x_n)^{\lambda/2}\Delta\right) = C(x_0\ldots x_n)^{\lambda/2}L\Delta = C(n+1)! (x_0\ldots x_n)^{\lambda/2}.
			\intertext{Now let us transform the left-hand side.
			Note that we alternate over~$\mathfrak{S}_{n+1}$ applying~$L$:}
			L\left(
				\sum_{k=0}^n(-1)^{kn} x_k^{\lambda(n+1)/2}\prod_{i\ne k}(ax_k + bx_i)\prod_{i < j \atop i,j\ne k}(x_i-x_j)
			\right) =
			\intertext{Therefore we can take out of the brackets alternation over a subgroup of~$\mathfrak{S}_{n+1}$.
			We take out alternation over the subgroup generated by the cyclic permutation~$\theta = (0,1,\ldots,n)$:}
			= L \sum_{k=0}^n(-1)^{k\theta} \theta^k
			(n+1) L
			\left(
				x_0^{\lambda(n+1)/2} \prod_{i\ne 0}(ax_0 + bx_i)\prod_{0 < i < j}(x_i-x_j)
			\right) =
			\intertext{Take the alternation over~$\langle\theta\rangle\subset\mathfrak{S}_{n+1}$ into~$L$ and get the alternation over the subgroup~$\mathfrak{S}_n$ permuting~$x_1,\ldots,x_n$ out of the inner product:}
			= (n+1)L
			\left(
				x_0^{\lambda(n+1)/2} \prod_{i\ne 0}(ax_0 + bx_i) \sum_{\theta\in\mathfrak{S}_n}(-1)^\theta\theta(x_1^{n-1}\ldots x_{n-1})
			\right) =
			\intertext{Denote by~$\sigma_i$ the elementary symmetric polynomials in~$x_1,\ldots,x_n$.
			Note that alternation over a group commutes with multiplication by a polynomial that is invariant under the action of this group:}
			(n+1)L \sum_{\theta \in \mathfrak{S}_n}(-1)^\theta \theta
			\left(
				x_0^{\lambda(n+1)/2}x_1^{n-1}\ldots x_{n-1} \cdot
				\sum_{k=0}^n \binom{n}{k}b^ka^{n-k}\sigma_k x_0^{n-k}
			\right) =
			\intertext{Now we get}
			= (n+1)! L
			\left(
				\sum_{k=0}^n \binom{n}{k} b^k a^{n-k} \cdot
				x_0^{n-k + \frac{\lambda}{2}(n+1)} x_1^{n-1}\ldots x_{n-1} \cdot\sigma_k
			\right) =
			\intertext{Recall that~$L(F)=0$ for any expression~$F$ that is symmetric with respect to transposition of any pair of variables.
			So we can drop all the monomials where the degrees are not pairwise different.
			In the product~$x_0^*x_1^{n-1}\ldots x_{n-1}\sigma_k$ the unique monomial with pairwise different degrees comes from the summand~$x_1\ldots x_k$ in~$\sigma_k$.}
			= (n+1)!L\left(
				\sum_{k=0}^n \binom{n}{k}b^ka^{n-k} x_0^{n+1-k + \frac{\lambda}{2}(n+1)}x_1^nx_2^{n-1}\ldots x_k^{n+1-k} x_{k+1}^{n-k-1}\ldots x_{n-1}
			\right) =
			\intertext{This expression equals the right-hand side of the initial equation:}
			= (n+1)!C(x_0\ldots x_n)^{\lambda/2} = (n+1)!CL(x_0^{n+\lambda/2}\ldots x_n^{\lambda/2}).
		\end{gather*}
		In the remaining part of the proof we deal with equality of the last two formulas:
		\begin{equation}
			\label{eq:HS}
			L\left(
				\sum_{k=0}^n \binom{n}{k}b^ka^{n-k}
				x_0^{n-k + \lambda/2(n+1)} x_1^{n-1}\ldots x_{n-1} \cdot x_1\ldots x_k
			\right)
			= C (x_0\ldots x_n)^{\lambda/2}
		\end{equation}

		Suppose~$b=0$.
		Since~$a+b=2$, we have~$a\ne 0$.
		In this case the function~$f_{\rm sing}(t)$ has the form
		\[
			f_{\rm sing}(t) = \frac{e^{\lambda t}}{\cth t - 1}.
		\]
		Then the left-hand side equals~$L\left(a^n x_0^{n+\lambda/2(n+1)}x_1^{n-1}\ldots x_{n-1}\right)$.
		One of the following two cases holds:
		\begin{itemize}
			\item[$C = 0$]
				Then~$n+\frac{\lambda}{2}(n+1)=k \in\{0,\ldots,n-1\}$ and~$\lambda = 2\frac{k-n}{n+1}$ for~$k\in\{0\ldots,n-1\}$.
				Note that we can pass from function~$f(t)$ being a solution of~\eqref{eq:H} to~$f(t+t_0)$ for any~$t_0$.
				The only condition on the argument shift is that~$f(t+t_0)=t+O(t^2)$.
				In this way we obtain
				\[
					f(t) = \frac{1}{\mu}e^{\lambda\mu t}\sinh \mu t,
				\]
				where~$\lambda=\frac{2k}{n+1}$ and~$k\in\mathbb{Z}$.
				Therefore~$f(t)$ is a degeneration of elliptic function of level~$\frac{n+1}{\operatorname{gcd}(k,n+1)}$.
			\item[$C\ne0$]
				Then
				\[
					L\left(
						a^n x_0^{n+\frac{\lambda}{2}(n+1)}x_1^{n-1}\ldots x_{n-1} - C x_0^{n+\frac{\lambda}{2}}\ldots x_n^{\frac{\lambda}{2}}
					\right) = 0.
				\]
				Minimal degrees of two monomials in brackets coincide,~i.\,e. either~$\frac{\lambda}{2}=0$ or~$n+\frac{\lambda}{2}(n+1)=\frac{\lambda}{2}$,~i.\,e.~$\lambda=0$ or~$-2$.
				These functions are particular cases of~$f_{\rm Td}$.
		\end{itemize}

		The case~$a=0$ can be carried out in the same way.
		As it noted above,~$a$ and~$b$ are symmetric, and their transposition corresponds to the replacement~$\lambda\mapsto -\lambda$.
		The case~$\lambda = 0$ corresponds to~\eqref{f:Todd}.
		In the case~$\lambda=2$ we have that~$f_{\rm sing}$ is a particular case of~\eqref{f:Todd} for~$\alpha=-2$,~$\beta=0$.

		It remains to check the case~$a,b\ne 0$.
		All non-zero images of monomials in~\eqref{eq:HS} under the map~$L$ are linearly independent.
		It follows directly from the fact the sets of degrees can not be obtained from each other by permutations.
		It can be easy shown by induction over~$n$.

		Let us right down sets of degrees of monomials in the left-hand side of~\eqref{eq:HS}:
		\begin{subequations}
			\begin{gather}
				\label{lines1}
				\begin{aligned}
					n  &- \frac{\lambda}{2}(n+1), & n-1,& & n-2,& &\ldots,& & 1,& &0 \\
					n-1&- \frac{\lambda}{2}(n+1), & n,  & & n-2,& &\ldots,& & 1,& &0 \\
					n-2&- \frac{\lambda}{2}(n+1), & n,  & & n-1,& &\ldots,& & 1,& &0 \\
					&& \ldots \\
					1  &- \frac{\lambda}{2}(n+1), & n,  & & n-1,& &\ldots,& & 2,& &0 \\
					   &- \frac{\lambda}{2}(n+1), & n,  & & n-1,& &\ldots,& & 2,& &1 \\
				\end{aligned}
				\intertext{The set of degrees of the right-hand side is}
				\label{lines2}
				n+\frac{\lambda}{2}, n-1+\frac\lambda2, \ldots, 1+\frac\lambda2, \frac\lambda2 \
			\end{gather}
		\end{subequations}

		For~$\lambda=0$ any~$f_{\rm sing}$ is a solution of~\eqref{eq:H}.
		So we can assume that~$\lambda\ne 0$.
		If~$\lambda\not\in\mathbb{R}$, then the set of numbers~$\frac{\lambda}{2},\ldots,n+\frac{\lambda}{2}$ can not coincide with the set~$n+\frac{\lambda}{2}(n+1),n-1,\ldots,0$.
		In this case the equation does not hold.
		So we can assume that~$\lambda\in\mathbb{R}$.

		Let~$\lambda > 0$.
		The last row in~\eqref{lines1} can not be obtained from others by permutations, because its minimal element~$-\frac{\lambda}{2}(n+1)$ is less than minimal elements of other rows.
		All other rows contain element~$0$.
		So we can remove~$0$ from them and use the same argument for the remaining rows.

		Let~$\lambda < 0$.
		The same arguments work inductively, if we take maximal elements, go from the first line to the last consequently eliminating~$n,n-1,\ldots,0$.

		We see that the images of monomials in left-hand side of~\eqref{eq:HS} are linearly independent Schur--Puiseux functions or zeros.
		For each monomial~$P$ we have~$L(P)\in \langle(x_0\ldots x_n)^{\lambda/2}\rangle$.
		Therefore each row in~\eqref{lines1} the corresponding number~$k-\frac{\lambda}{2}(n+1)$ coincides with some other, or the row coincides with~\eqref{lines2} up to a permutation.

		Suppose some row in~\eqref{lines1} coincides with~\eqref{lines2} up to permutation.
		Then~$\lambda=0,\pm2$.
		Suppose that in each row in~\eqref{lines1} there is a pair of equal numbers.
		Then~$-\frac{\lambda}{2}(n+1),\ldots,n-\frac{\lambda}{2}(n+1) \in \{0,\ldots,n\}$.
		This means that~$\lambda=0$.
	\end{proof}

	This completes the classification of $n$-multiplicative Krichever genera.

\bibliographystyle{unsrt}

\addtocontents{toc}{\protect\setcounter{tocdepth}{-2}}

\end{document}